\documentclass{amsart}
\usepackage{amsmath}
\usepackage{marginnote}
\usepackage[top=1.5cm, bottom=1.5cm, outer=5cm, inner=2cm, heightrounded, marginparwidth=2.5cm, marginparsep=2cm]{geometry}
\usepackage{graphicx} 
\usepackage{tikz}
\usetikzlibrary{positioning, shapes, shapes.multipart, arrows.meta}
\usepackage{amssymb}
\usepackage{amsthm, multicol, color, wasysym}

\newtheorem{theorem}{Theorem}
\newtheorem{corollary}{Corollary}
\newtheorem{proposition}{Proposition}
\newtheorem{lemma}{Lemma}
\newtheorem{conjecture}{Conjecture}
\newtheorem{observation}{Observation}

\newcommand{\spec}{\mathrm{spec}}
\newcommand{\tr}{\mathrm{tr}}
\newcommand{\ds}{\displaystyle}

\newcommand{\trans}{t}

\title{Entropy of Tournament Digraphs}
\author{David E. Brown}
\author{Eric Culver}
\author{Bryce Frederickson}
\author{Sidney Tate}
\author{Brent J. Thomas}

\begin{document}

\begin{abstract}
The R\'{e}nyi $\alpha$-entropy $H_{\alpha}$ of complete antisymmetric directed graphs (i.e., tournaments) is explored.  
We optimize $H_{\alpha}$ when $\alpha = 2$ and $3$, and find that as $\alpha$ increases $H_{\alpha}$'s sensitivity to what we refer to as `regularity' increases as well. 
A regular tournament on $n$ vertices is one with each vertex having out-degree $\frac{n-1}{2}$, but there is a lot of diversity in terms of structure among the regular tournaments; for example, a regular tournament may be such that each vertex's out-set induces a regular tournament (a doubly-regular tournament) or a transitive tournament (a rotational tournament).  As $\alpha$ increases, on the set of regular tournaments, $H_{\alpha}$ has maximum value on doubly regular tournaments and minimum value on rotational tournaments.  The more `regular', the higher the entropy.  We show, however, that $H_2$ and $H_3$ are maximized, among all tournaments on any number of vertices by any regular tournament.  
We also provide a calculation that is equivalent to the von Neumann entropy, but may be applied to any directed or undirected graph and shows that the von Neumann entropy is a measure of how quickly a random walk on the graph or directed graph settles. 
\end{abstract}

\maketitle
\section{Introduction}

We present results about an entropy function applied to directed graphs, in particular to orientations of complete graphs --- also known as \emph{tournaments}.
While there is a fair amount of recent research focusing on entropy applied to undirected graphs,
there is not as much applied to directed graphs in spite of the fact that many real-world networks 
such as citation, communication, financial and neural are best modeled with directed graphs.

All graphs are finite and simple. 
The degree of vertex $v$ in an undirected graph $G$ will be denoted $\deg_G(v)$
(subscripts omitted if the context allows), we write $V(G)$ for the vertex set of graph $G$, $E(G)$ 
for the adjacency relation of $G$, and we write $xy \in E(G)$ to indicate that vertices $x,y \in V(G)$ are adjacent 
in $G$.  
If $G$ is a directed graph we will use $A(G)$ to denote the adjacency relation since we may refer to elements of 
$A(G)$ as \emph{arcs}, and write $x \to y \in A(G)$ or $x \to y$ in $G$ to denote the arc \emph{from} $x$ \emph{to} $y$ in $G$; 
$x$ is the \emph{tail} of arc $x \to y$ and $y$ is the \emph{head}.
For a directed graph $G$ and $x \in V(G)$, the set $N^+_G(x) = \{y \in V(G): x \to y \in A(G)\}$ is called the \emph{out-set} of $x$ in $G$ (subscript omitted in appropriate contexts).
The nonnegative integer $|N^+_G(x)|$ is the \emph{out-degree} or \emph{score} of vertex $x$ in directed graph $G$ and will be denoted $d^+_G(v)$ (subscript omitted if context allows).
We use $M(i,j)$ to denote entry $(i,j)$ of matrix $M$, $\spec(M)$ to denote the spectrum of $M$ 
(the multiset of eigenvalues of $M$), and $\tr(M)$ to denote the trace of $M$ ($\tr(M) = \sum_{i}M(i,i)$). 
Other notation defined as needed. 

The entropy of an undirected graph has been defined in many ways, with many motivations, 
but the starting point for our investigation is the classical Shannon entropy that, with a sleight of hand, is applied to the spectrum of a matrix representing the graph's structure.
Many other functions intended to represent the entropy of undirected graphs that are in 
contradistinction to those we explore are surveyed in \cite{dehmer2011history}. 
The  Shannon entropy of a discrete probability distribution $\vec{p} = (p_1, \dots, p_n)$ is 
\begin{equation} S(\vec{p}) = \sum_{p_i \in \vec{p}}p_i \log_2\frac{1}{p_i}, \label{Shannon ent def} \end{equation} 
and $S(\vec{p})$ is intended to be a measure of the information content in messages transmitted over a channel 
in which bit $i$ occurs with probability $p_i$. 
In the field of quantum information theory the von Neumann entropy is used heavily; see \cite{nielsen2010quantum} and of course \cite{von2018mathematical}.  
The von Neumann entropy of a quantum state of a physical system is defined in terms of 
the eigenvalues of the \emph{density matrix} associated to the physical system.
The density matrix is Hermitian, positive semi-definite, and has unit trace. 
Hence the spectrum of the density matrix has the characteristics of a discrete probability distribution;
and thereby the entropy of the physical system is defined to be the Shannon entropy of the spectrum of the density matrix.    
Suppose $G$ is an undirected graph with $V(G) = \{v_1, \dots, v_n\}$.  
The Laplacian of $G$, denoted $L_G$, 
is the matrix with non-diagonal entry $(i,j) = -1$ if $v_iv_j \in E(G)$, $0$ otherwise, and diagonal entry $(i,i)$ equal to the degree of vertex $v_i$.  
Alternatively, we think of the Laplacian $L_G$ as $D_G - A_G$, where $A_G$ is the \emph{adjacency matrix} of $G$ 
($A_G(i,j) = 1$ if $v_iv_j \in E(G)$ and $0$ otherwise), and $D_G$ is the \emph{degree matrix} of $G$ 
($D_G(i,i) = \deg_G(v_i)$ and $D_G(i,j) = 0$ if $i \neq j$).
In this paper, we define the \emph{normalized Laplacian} matrix of $G$, by 
$\overline{L}_G = \frac{1}{\tr(L_G)}L_G$.
Note that $\overline{L}_G$ is symmetric, positive semi-definite, and has unit trace; 
therefore $\overline{L}_G$ may be thought of as the density matrix of a physical system with $G$ its representation as an undirected graph.  
The \emph{von Neumann entropy} of graph $G$, denoted $H(G)$, is the von Neumann entropy of $G$'s normalized Laplacian: 
\begin{equation} H(G) =  \sum_{\lambda \in \spec(\overline{L}_G)} \lambda \log_2 \frac{1}{\lambda}, \label{von Neumann ent def} \end{equation}
where $0 \log_2 \frac 10$ is conventionally taken to be $0$. 

The entropy of an undirected graph has been defined to be the von Neumann entropy of its normalized Laplacian 
by many authors and for many reasons, see \cite{anand2011shannon, braunstein2006laplacian, dairyko2017note, de2016interpreting, ye2016jensen}.
For example the von Neumann entropy's interpretation when applied to a graph is studied in \cite{de2016interpreting}, 
it is studied as a measure of network regularity in \cite{passerini1quantifying}, 
in the context of representing quantum information in \cite{belhaj2016weighted}, 
and in \cite{dairyko2017note} its connection to graph parameters among other things is studied. 
The variety of applications and interpretations in the aforementioned references, at least to some extent,
substantiates saying that it is not clear what entropy of a graph, in particular its von Neumann entropy, is telling us. 
This paper is a contribution to that conversation in the context of directed graphs.  

A directed graph's Laplacian, however, is not necessarily symmetric or positive semi-definite; 
consequently we cannot simply treat its spectrum as a discrete probability distribution.  
But, in this paper, we come to the entropy of a directed graph via a function  developed by R\'{e}nyi 
in \cite{renyi1961measures} to generalize Shannon's entropy: 
\begin{equation} H_{\alpha}(\vec{p}) = \frac{1}{1-\alpha}\log_2\left(\sum_{p_i \in \vec{p}}p_i^{\alpha}\right), \label{Renyi-alpha def}\end{equation}
where $\vec{p}$ is a discrete probability distribution as in the Shannon entropy, $\alpha >0$ and $\alpha \neq 1$.
Suppose $\Gamma$ is a directed graph with $V(\Gamma) = \{v_1,\dots, v_n\}$; 
the Laplacian of $\Gamma$, $L_{\Gamma}$, is constructed the same way as is 
the Laplacian of an undirected graph: $$L_{\Gamma}(i,j) = \left\{ \begin{array}{cc} d^+(v_i) & \mbox{ if $i = j$} \\ 
-1 & \mbox{ if $v_i \to v_j\in A$} \\ 0 & \mbox{ if $v_i \to v_j \not\in A$} \end{array} \right ..$$

We define, for directed graph $\Gamma$ with \emph{normalized} Laplacian $\overline{L}_{\Gamma}$ whose spectrum is $\Lambda_\Gamma$, 
its \emph{R\'{e}nyi $\alpha$-entropy} to be $H_{\alpha}(\Gamma) = H_{\alpha}(\Lambda_\Gamma)$.
Note that $S(\vec{p}) = \lim_{\alpha \to 1}H_{\alpha}(\vec{p})$ (see \cite{renyi1961measures}) but we 
focus on positive integer values of $\alpha$ greater than $1$;
doing this makes moot the inconvenient characteristics of the spectrum of a 
directed graph's Laplacian and also allows us to use combinatorial arguments to 
compute entropy. 
To wit, suppose $\Gamma$ is a directed graph whose normalized Laplacian is $\overline{L} = \frac{1}{\tr(D-A)}\left(D - A\right)$, 
where $A$ is $\Gamma$'s adjacency matrix, $D$ the diagonal matrix with out-degrees of vertices of $\Gamma$ as its diagonal entries, and 
$L$ its Laplacian; 
then using the various properties of the trace function\footnote{Recall the trace is linear, and that for any square matrix $M$, $\tr(M) = \sum_{i}M(i,i)
= \sum_{\lambda \in \spec(M)}\lambda$. 
Also, if $\lambda$ is an eigenvalue of $M$, then $\lambda^k$ is an eigenvalue of $M^k$
and so $\tr(M^k) = \sum_{\lambda \in \spec(M)}\lambda^k$, and 
$\tr(AB) = \tr(BA)$ for (in particular) square matrices $A$ and $B$.} and focusing on the argument of the logarithm, we have 
\begin{eqnarray*} \sum_{\lambda \in \Lambda_\Gamma} \lambda^2=  \tr\left(\overline{L}^2\right) &=& \tr\left(\left(\frac{1}{\tr(D-A)}(D -A)\right)^2\right) \\ &= &
\tr(D-A)^{-2}\left(\tr(D^2) - \tr(AD) - \tr(DA) + \tr(A^2)\right). \end{eqnarray*}
Noting that $A^{\alpha}$ records the number of walks of length $\alpha$ between vertices, we see that the computation of $H_{\alpha}(\Gamma)$ will involve $\Gamma$'s out-degree raised to powers and the number of walks of length $\alpha$ from vertices to themselves. 

A directed graph $T$ with $|V(T)| = n$ is an \emph{$n$-tournament} if  
for each pair of vertices $x,y \in V(T)$ we have \emph{either} $x \to y \in A(T)$ or $y\to x \in A(T)$; 
in other terms, an $n$-tournament is an orientation of the complete graph on $n$ vertices. 
Note that if $M$ is the adjacency matrix of an $n$-tournament, then $M + M^t = J_n - I_n$, where $J_n$ is the $n \times n$ 
matrix all of whose entries equal $1$, and $I_n$ is the $n \times n$ identity matrix.

Now suppose $\Gamma$ is an $n$-tournament, then the trace of its Laplacian is $\binom{n}{2}$, and there are no walks of length $2$ from any vertex to itself and so 
in the computation of $\tr\left(\overline{L}^{\alpha}\right)$, with 
$\alpha$ an integer greater than or equal to $2$, terms such as 
$\tr\left(D^{\alpha -2}A^2\right), \tr\left(AD^{\alpha-2}A \right)$,
and $\tr\left(A^2D^{\alpha-2}\right)$ equate to zero. 

More generally, we have the following result we will use in the sequel and which follows from the same properties of the trace used above and those of tournaments.  

\begin{lemma}\label{lem:powerequaltrace}
Suppose $L = g(D -A)$ is the normalized Laplacian of an $n$-tournament \emph{(so $g = \binom{n}{2}^{-1}$)}, and let  
$\Lambda=\spec(L)$, then 
\[\sum_{\lambda \in \Lambda} \lambda^3 = \tr\left( g^3(D-A)^3 \right)=g^3\left(\tr\left(D^{3}\right)-\tr\left(A^3\right)\right),\]
and
\[\sum_{\lambda \in \Lambda} \lambda^4 = \tr\left( g^4(D-A)^4 \right)=g^4\left(\tr\left(D^{4}\right)-\tr\left(DA^3\right) - \tr\left(A^3D\right) + \tr\left(A^4\right) \right).\]
\end{lemma}

An $n$-tournament is \emph{regular} if the score of each vertex is $\frac{n-1}{2}$. 
The number of $3$-cycles in a labeled $n$-tournament $T$  with $V(T) = \{v_1, \dots, v_n\}$ 
is obtained via \begin{eqnarray}\binom{n}{3} - \sum_{1 \leq i \leq n} \binom{d^+(v_i)}{2} \label{number_of_3-cycles}\end{eqnarray}
and this number is maximized when $T$ is regular (and $n$ is necessarily odd).

On the other hand an $n$-tournament $T$ has no cycles if and only if it $T$ transitive: $T$ is \emph{transitive} 
if, for all $x,y,z \in V(T)$, $x \to y$ and $y \to z$ implies $x \to z$. 
Also, an $n$-tournament is transitive if and only if its vertices can be labeled $v_0, v_1, \dots, v_{n-1}$ 
so that $d^+(v_i) = i$; that is, its \emph{score sequence} is $(0,1,2, \dots, n-1)$. 
There is one transitive $n$-tournament up to isomorphism for each integer $n \geq 1$. 
In contrast, up to isomorphism, there are 1,123 and 1,495,297 regular $11$-tournaments and regular $13$-tournaments, respectively. 
We will show that for $\alpha = 2,3$ and $n > 4$, the transitive and regular $n$-tournaments 
yield minimum and maximum R\'{e}nyi $\alpha$-entropy, respectively. 
But this is reductive in the case of regular $n$-tournaments, for $n > 5$; 
the R\'{e}nyi $\alpha$-entropy distinguishes among regular tournaments and gives a continuum of `\emph{regularity}' -- for lack of a better term. 
If $n$ is odd, then for $\alpha = 2$ and $\alpha = 3$, $H_{\alpha}(T)$ is minimum on the set of $n$-tournaments if and only if 
$T$ is transitive; $H_{\alpha}(T)$ is maximum if and only if $T$ is regular.

\subsection{Small Tournaments}
Let $\mathcal{T}_n$ denote the set of all $n$-tournaments up to isomorphism.
In the hope of shedding light on what the R\'{e}nyi entropy is telling us, and to foreshadow sequel sections, we examine the R\'{e}nyi entropy's behavior on $\mathcal{T}_4$, $\mathcal{T}_5$, and $\mathcal{T}_3$.

Up to isomorphism there are $4$ distinct $4$-tournaments.  
The \emph{score sequence} of an $n$-tournament on vertices $v_1, \dots, v_n$ is the list 
$(s_1, \dots, s_n)$ with, relabeling if necessary, $s_i = d^+(v_i)$ and $s_1 \leq s_2 \leq \cdots \leq s_n$.
The $4$-tournament $TS_4$ in Figure \ref{fig:4_tournaments} represents the isomorphism class of all $4$-tournaments with score sequence $(1,1,2,2)$.  
The other isomorphism classes of $4$-tournaments are determined by their score sequences (this is the case only for $n$-tournaments with $n \leq 4$); 
the other $4$-tournament score sequences are $(0,2,2,2)$, $(1,1,1,3)$, and $(0,1,2,3)$, which have $TK_4$, $TO_4$, and $TT_4$, respectively, 
as their associated tournaments. 

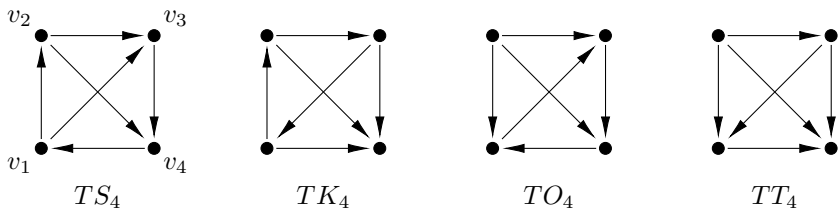
\begin{figure}[ht]
\begin{center}
	\begin{tikzpicture}[scale = 1.5, >={Latex[length=3mm, width=1.25mm]}]
			\foreach \x/\y/\z in {0/0/a, 0/1/b, 1/0/d, 1/1/c, 2/0/e, 3/0/h, 4/0/i, 5/0/l, 6/0/m, 7/0/p,
			      2/1/f, 3/1/g, 4/1/j, 5/1/k, 6/1/n, 7/1/o}{
			\draw[fill] (\x, \y) circle (1.5pt) node (\z) {};
			}
			\draw[->] (a) -- (b);
			\draw[->] (b) -- (d);
			\draw[->] (b) -- (c);
			\draw[->] (c) -- (d);
			\draw[->] (d) -- (a); 
			\draw[->] (a) -- (c);
			
			\draw[->] (e) -- (f);
			\draw[->] (f) -- (g);
			\draw[->] (g) -- (e);
			\draw[->] (e) -- (h);
			\draw[->] (f) -- (h); 
			\draw[->] (g) -- (h);
			
			\draw[->] (j) -- (i);
			\draw[->] (j) -- (k);
			\draw[->] (j) -- (l);
			\draw[->] (i) -- (k);
			\draw[->] (k) -- (l); 
			\draw[->] (l) -- (i);
			
			\draw[->] (n) -- (o);
			\draw[->] (n) -- (p);
			\draw[->] (n) -- (m);
			\draw[->] (o) -- (m);
			\draw[->] (o) -- (p); 
			\draw[->] (m) -- (p);
			
			\node [below left] at (a)  {$v_1$}; 
			\node [above left] at (b)  {$v_2$};
			\node [above right] at (c)  {$v_3$};
			\node [below right] at (d)  {$v_4$};
			\node [below] at (0.5, -0.25) {$TS_4$};
			
			\node[below] at (2.5,-0.25) {$TK_4$};
			
			\node[below] at (4.5,-0.25) {$TO_4$};
			
			\node[below] at (6.5,-0.25) {$TT_4$};

	\end{tikzpicture}
\end{center}
\caption{All $4$-tournaments}\label{fig:4_tournaments}
\end{figure}

By Lemma \ref{lem:powerequaltrace}  
\begin{eqnarray*}H_2(TS_4) & =&  -\log_2\left(\tr(\overline{L}_T)^2\right)
= -\log_2\left(\tr\left(\frac{1}{36}(D - A)^2\right)\right)\\ & = &-\log_2\left(\frac{1}{36}\left(\tr(D^2) - 2\tr(DA) + \tr(A^2)\right)\right), \end{eqnarray*}   
and since no vertex of a tournament has a walk of length $2$ from itself to itself, 
the trace of its adjacency matrix squared is zero.  Also, $\tr(DA) = \tr(AD) = 0$.
Therefore, $H_2(TS_4) = -\log_2\left(\tr(D^2)/36\right) = -\log_2\left(\sum_{1 \leq i \leq 4}(d^+(v_i))^2/36\right)=-\log_2\left(\left(1^2+1^2+2^2+2^2\right)/36\right)$.
Indeed, for any $n$-tournament $T$ on vertices $v_1, \dots, v_n$, 
$$H_2(T)=-\log_2\left(\binom{n}{2}^{-2}\sum_{1\leq i \leq n} d^+(v_i)^2\right).$$
With $\alpha =3$, the calculation is $$H_3(T) = -\log_2\left(\binom{n}{2}^{-3}\sum_{1\leq i \leq n} d^+(v_i)^3 -  \sum_{1\leq i \leq n}c_3(i,i) \right),$$
where $c_3(i,j)$ is the number of walks of length $3$ from $v_i$ to $v_j$.

The table at (\ref{table:H2&H3on4tournies}) displays essentially $H_2$ and $H_3$ for all $4$-tournaments; in fact $\sum_{\lambda \in \spec(L_T)} \lambda^{\alpha}$, for $\alpha = 2, 3$ and each $T \in \mathcal{T}_4$ are displayed. 

\begin{equation}
\begin{array}{| c | c | c |}
\hline 
 & \sum \lambda^2 & \sum \lambda^3 \\
\hline
TS_4 & 10 & 12 \\
TK_4 & 12 & 21 \\
TO_4 & 12 & 27 \\
TT_4 & 14 & 36 \\

\hline
\end{array}\label{table:H2&H3on4tournies}
\end{equation}

Though both $H_2$ and $H_3$ are functions only of the score sequence, $H_3$ seems to quantify something more than $H_2$ does, and distinguishes each tournament in $\mathcal{T}_4$.  

We now explore $\mathcal{T}_5$.  
There are $12$ distinct $5$-tournaments up to isomorphism and $9$ distinct score sequences.  
The score sequences $(1,2,2,2,3)$ and $(1,1,2,3,3)$ have $3$ and $2$ distinct tournaments associated with them, see Figure \ref{fig:5_tournaments_(1,2,2,2,3)} and Figure \ref{fig:5_tournaments_(1,1,2,3,3)}. 

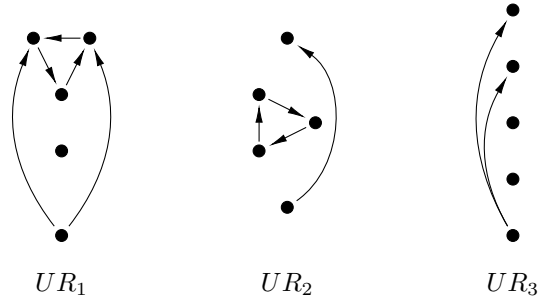
\begin{figure}[ht]
\begin{center}
	\begin{tikzpicture}[scale = 1.5, >={Latex[length=2.5mm, width=1.0mm]}]
			\foreach \x/\y/\z in {4/0/e, 4/0.5/d, 4/1/c, 4/1.5/b, 4/2/a,
			0/0/j, 0/0.75/i, 0/1.25/h, -0.25/1.75/f, 0.25/1.75/g,
			2/0.25/o, 1.75/0.75/m, 1.75/1.25/n, 2.25/1/l, 2/1.75/k}{
			\draw[fill] (\x, \y) circle (1.5pt) node (\z) {};
			}
		
			\draw[bend right,<-] (a) to node [auto] {} (e);
			\draw[bend right,<-] (b) to node [auto] {} (e);
			\draw[bend left,->] (j) to node [auto] {} (f);
			\draw[bend right,->] (j) to node [auto] {} (g);
			\draw[->] (g) -- (f);
			\draw[->] (f) -- (h);
			\draw[->] (h) -- (g);
			\path[->] (o) edge [bend right = 55] node {} (k);
			\draw[->] (l) -- (m);
			\draw[->] (m) -- (n);
			\draw[->] (n) -- (l);
			
			\node [below] at (0, -0.25) {$UR_1$};
			\node [below] at (2,-0.25) {$UR_2$};
			\node [below] at (4, -0.25) {$UR_3$};
			
	\end{tikzpicture}
\end{center}
\caption{Non-isomorphic $5$-tournaments with score sequence $(1,2,2,2,3)$. Arcs not depicted are directed downward.}\label{fig:5_tournaments_(1,2,2,2,3)}
\end{figure}

\begin{figure}[ht]
\begin{center}
	\begin{tikzpicture}[scale = 1.5, >={Latex[length=2.5mm, width=1.0mm]}]
			\foreach \x/\y/\z in {2/0/e, 2/0.5/d, 2/1/c, 2/1.5/b, 2/2/a,
			0/0/j, 0.25/0.5/i, -0.25/0.5/h, 0/1/g, 0/1.5/f}{
			\draw[fill] (\x, \y) circle (1.5pt) node (\z) {};
			}
		
			\draw[bend right,<-] (a) to node [auto] {} (e);
			\path[->] (j) edge [bend right = 55] node {} (f);
			\draw[->] (h) -- (i);
			\draw[->] (i) -- (j);
			\draw[->] (j) -- (h);
			
			\node [below] at (0,-0.25) {$U_1$};
			\node [below] at (2,-0.25) {$U_2$};
			
	\end{tikzpicture}
\end{center}
\caption{Non-isomorphic $5$-tournaments with score sequence $(1,1,2,3,3)$. Arcs not depicted are directed downward.}\label{fig:5_tournaments_(1,1,2,3,3)}
\end{figure}
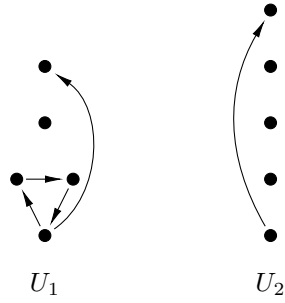

Table \ref{table:H2&H3&H4on5tournies} shows the R\'{e}nyi $\alpha$-entropy values for all the $5$-tournaments, for $\alpha = 2,3,4$.  
Actually, again, what is shown is $\sum_{\lambda \in \spec(L_T)}\lambda^{\alpha}$, $\alpha = 2,3,4$, and $T \in \mathcal{T}_5$. 
We use $T_{\vec{s}}$ to denote the (unique in this case) tournament corresponding to the score sequence $\vec{s}$.  $TT_5$ is the transitive $5$-tournament, $R_5$ is the $5$-tournament with score sequence $(2,2,2,2,2)$.  

\begin{equation}
\begin{array}{| c | c | c | c |}
\hline 
 & \sum \lambda^2 & \sum \lambda^3 & \sum \lambda^4  \\
\hline
R_5 & 20 & 25  & -20\\
UR_1 & 22 & 40 & 46\\
UR_2 & 22 & 40 & 46\\
UR_3 & 22 & 40 & 50 \\
U_1 & 24 & 55 & 116 \\
U_2 & 24 & 55 & 120 \\
T_{(0,2,2,3,3)} \; (\mbox{``$E$''}) & 26 & 76 & 258 \\
T_{(1,1,2,2,4)} \; (\mbox{``$D$''}) & 26 & 64 & 138 \\
T_{(1,1,1,3,4)} \; (\mbox{``$C$''}) & 28 & 79 & 208 \\
T_{(0,2,2,2,4)} \; (\mbox{``$B$''}) & 28 & 85 & 280 \\
T_{(0,1,3,3,3)} \; (\mbox{``$A$''}) & 28 & 91 & 328 \\
TT_5 & 30 & 100 & 354 \\
 
\hline
\end{array}\label{table:H2&H3&H4on5tournies}
\end{equation}
Notice that as $\alpha$ increases the number of distinct entropy values increases.  
Consider the partial order induced by the R\'{e}nyi entropy, 
where $T_1 <_{\alpha} T_2$ if $H_{\alpha}(T_1) < H_{\alpha}(T_2)$. 
Figure \ref{fig:Hasse_of_H2H3H4_on_5tourns} shows the Hasse diagrams for the orders $<_i$, for $i = 2,3,4$.
We see fewer incomparabilities as $\alpha$ increases, but $<_{i}$ is not necessarily a refinement of $<_{i-1}$. For example, $C <_2 E$, $C <_3 E$, but $E <_4 C$.

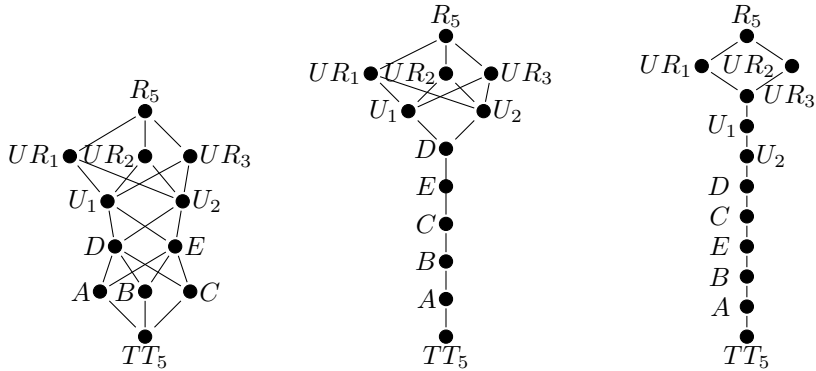
\begin{figure}[ht]
\begin{center}
	\begin{tikzpicture}[scale = 2, >={Latex[length=4mm, width=1.5mm]}]
			\foreach \x/\y/\z in {0/0/TT5, 0/0.3/B, -0.3/0.3/A, 0.3/0.3/C, -0.2/0.6/D, 0.2/0.6/E,
			-0.25/0.9/U1, 0.25/0.9/U2, -0.5/1.2/UR1, 0/1.2/UR2, 0.3/1.2/UR3, 0/1.5/R5,
			2/0/TT51, 2/0.25/A1, 2/0.5/B1, 2/0.75/C1, 2/1/E1, 2/1.25/D1, 1.75/1.5/U11, 2.25/1.5/U21, 
			1.5/1.75/UR11, 2/1.75/UR21, 2.3/1.75/UR31, 2/2/R51,
			4/0/TT52, 4/0.2/A2, 4/0.4/B2, 4/0.6/E2, 4/0.8/C2, 4/1.0/D2, 4/1.2/U22, 4/1.4/U12, 
			4/1.6/UR32, 3.7/1.8/UR12, 4.3/1.8/UR22, 4/2/R52}{
			\draw[fill] (\x, \y) circle (1.25pt) node (\z) {};
			}

			\node[below] at (TT5) {$TT_5$};
			\node[left] at (A) {$A$};
			\node[left] at (B) {$B$};
			\node[right] at (C) {$C$};
			\node[left] at (D) {$D$};
			\node[right] at (E) {$E$};
			\node[left] at (U1) {$U_1$};
			\node[right] at (U2) {$U_2$};
			\node[left] at (UR1) {$UR_1$};
			\node[left] at (UR2) {$UR_2$};
			\node[right] at (UR3) {$UR_3$};
			\node[above] at (R5) {$R_5$};
			
			\node[below] at (TT51) {$TT_5$};
			\node[left] at (A1) {$A$};
			\node[left] at (B1) {$B$};
			\node[left] at (C1) {$C$};
			\node[left] at (D1) {$D$};
			\node[left] at (E1) {$E$};
			\node[left] at (U11) {$U_1$};
			\node[right] at (U21) {$U_2$};
			\node[left] at (UR11) {$UR_1$};
			\node[left] at (UR21) {$UR_2$};
			\node[right] at (UR31) {$UR_3$};
			\node[above] at (R51) {$R_5$};
			
			\node[below] at (TT52) {$TT_5$};
			\node[left] at (A2) {$A\;$};
			\node[left] at (B2) {$B\;$};
			\node[left] at (C2) {$C\;$};
			\node[left] at (D2) {$D\;$};
			\node[left] at (E2) {$E\;$};
			\node[left] at (U12) {$U_1$};
			\node[right] at (U22) {$U_2$};
			\node[left] at (UR12) {$UR_1$};
			\node[left] at (UR22) {$UR_2\;$};
			\node[right] at (UR32) {$\;UR_3$};
			\node[above] at (R52) {$R_5$};
			
			\draw[-] (TT5) -- (B);
			\draw[-] (TT5) -- (A);
			\draw[-] (TT5) -- (C);
			\draw[-] (A) -- (D);
			\draw[-] (A) -- (E); 
			\draw[-] (B) -- (D);
			\draw[-] (B) -- (E);
			\draw[-] (C) -- (E); 
			\draw[-] (C) -- (D);
			\draw[-] (D) -- (U1);
			\draw[-] (D) -- (U2); 
			\draw[-] (E) -- (U1);
			\draw[-] (E) -- (U2);
			\draw[-] (U1) -- (UR1);
			\draw[-] (U1) -- (UR2);
			\draw[-] (U1) -- (UR3);
			\draw[-] (U2) -- (UR1);
			\draw[-] (U2) -- (UR2);
			\draw[-] (U2) -- (UR3);
			\draw[-] (UR1) -- (R5);
			\draw[-] (UR2) -- (R5);
			\draw[-] (UR3) -- (R5);
			
			\draw[-] (TT51) -- (A1);
			\draw[-] (A1) -- (B1);
			\draw[-] (B1) -- (C1);
			\draw[-] (C1) -- (D1);
			\draw[-] (D1) -- (U11);
			\draw[-] (D1) -- (U21);
			\draw[-] (U11) -- (UR11);
			\draw[-] (U11) -- (UR21);
			\draw[-] (U11) -- (UR31);
			\draw[-] (U21) -- (UR11);
			\draw[-] (U21) -- (UR21);
			\draw[-] (U21) -- (UR31);
			\draw[-] (UR11) -- (R51);
			\draw[-] (UR21) -- (R51);
			\draw[-] (UR31) -- (R51);
			
			\draw[-] (TT52) -- (A2);
			\draw[-] (A2) -- (B2);
			\draw[-] (B2) -- (E2);
			\draw[-] (E2) -- (C2);
			\draw[-] (C2) -- (D2);
			\draw[-] (D2) -- (U22);
			\draw[-] (U22) -- (U12);
			\draw[-] (U12) -- (UR32);
			\draw[-] (UR32) -- (UR12);
			\draw[-] (UR32) -- (UR22);
			\draw[-] (UR22) -- (R52);
			\draw[-] (UR12) -- (R52);

	\end{tikzpicture}
\end{center}
\caption{Hasse diagrams of the partial orders determined by $H_2$, $H_3$, and $H_4$.}\label{fig:Hasse_of_H2H3H4_on_5tourns}
\end{figure}

We now compare the R\'enyi $\alpha$-entropy of the two distinct $3$-tournaments as a function of $\alpha$ -- in what remains of this section $\alpha$ is not necessarily an integer.  
We treat this case last (out of $n = 3,4,5)$ because it is a bit different, but the results are consistent with the over arching claims we make about the R\'{e}nyi $\alpha$-entropy: that it is a measure of how regular a tournament is; the higher the entropy value, the more regular the tournament is.  
Moreover, and this will not be shown until the penultimate section there is more to `regular tournaments' than score sequences.

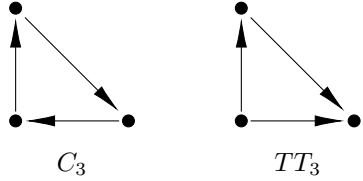
\begin{figure}[ht]
\begin{center}
	\begin{tikzpicture}[scale = 1.5, >={Latex[length=4mm, width=1.5mm]}]
			\foreach \x/\y/\z in {0/0/a, 0/1/b, 1/0/c, 2/0/d, 2/1/e, 3/0/f}{
			\draw[fill] (\x, \y) circle (1.5pt) node (\z) {};
			}
			
			\draw[->] (a) -> (b);
			\draw[->] (b) -- (c);
			\draw[->] (c) -- (a);
			\draw[->] (d) -- (e);
			\draw[->] (e) -- (f);
			\draw[->] (d) -- (f);
			
			\node [below] at (0.5,-0.2) {$C_3$};
			\node [below] at (2.5,-0.2) {$TT_3$};
		
		\end{tikzpicture}
\end{center}
\caption{The $3$-tournament which is a cycle (left) and the transitive $3$-tournament (right)}\label{fig:3_tourns}
\end{figure}

With $C_3$ denoting the $3$-tournament that is a cycle,
\begin{equation*}
\overline{L}(C_3) = \begin{bmatrix}
\frac{1}{3} & -\frac{1}{3} & 0 \\
0 & \frac{1}{3} & -\frac{1}{3} \\
-\frac{1}{3} & 0 & \frac{1}{3}
\end{bmatrix} \text{ has spectrum } \left\{0, \frac{1}{2} \pm \frac{\sqrt{3}}{6}i\right\} \text{, so}
\end{equation*}
\begin{align*}
H_\alpha(C_3) &= \frac{1}{1 - \alpha} \log_2 \left(2\left(\frac{1}{\sqrt{3}}\right)^\alpha \cos \left(\frac{\pi}{6}\alpha\right)\right) \\
&= \frac{1 - \alpha \log_2 \sqrt 3}{1 - \alpha} + \log_2 \left(\cos^{\frac{1}{1 - \alpha}} \left(\frac{\pi}{6}\alpha\right)\right).
\end{align*}
Now consider the domain for which this function gives a real-valued entropy.
If the cosine evaluates to 0, as is the case for $\alpha = 3$, then $H_\alpha$ is not defined, 
and we see a vertical asymptote as $\alpha \to 3$.
If the cosine value is negative, then the value of $S_\alpha$ is real only if $\alpha$ is of the form $2p/(2q +1)$ with $p,q \in \mathbb Z$.

As far as end behavior, $H_\alpha$ has no limit as $\alpha$ approaches infinity, but it does have a lower bound. 
We note that $H_\alpha$ has local minima at or near $12k$, with $k \in \mathbb{Z}^+$.
Then
\begin{align*}
H_{12k}(C_3) &= \frac{1-12k \log_2 \sqrt 3}{1-12k} + \frac{1}{1 - 12k}\log_2\left(\cos(2\pi k)\right)\\
&= \frac{1-12k \log_2 \sqrt 3}{1-12k}.
\end{align*}
As $k \to \infty$, the first term tends to $0$, and
$$\lim_{k \to \infty} H_{12k}(C_3) = \log_2 \sqrt{3} \approx 0.7925.$$

With $TT_3$ denoting the transitive $3$-tournament, we see that all eigenvalues are real, and the entropy is more well-behaved.
\begin{equation*}
\overline{L}(TT_3) = \begin{bmatrix}
\frac{2}{3} & -\frac{1}{3} & -\frac{1}{3} \\
0 & \frac{1}{3} & -\frac{1}{3} \\
0 & 0 & 0
\end{bmatrix} \text{ has spectrum } \left\{0, \frac{1}{3}, \frac{2}{3}\right\}, \text{ so}
\end{equation*}
\begin{align*}
H_\alpha(TT_3) &= \frac{1}{1 - \alpha} \log_2 \left(\left(\frac{1}{3}\right)^\alpha + \left(\frac{2}{3}\right)^\alpha \right).
\end{align*}
This function is continuous on $(1, \infty)$, and we can evaluate $\lim_{\alpha \to \infty}H_{\alpha}(T)$ 
by applying L'H\^opital's Rule (when the base is not specified `$\log$' is the natural logarithm):
\begin{align*}
\lim_{\alpha \to \infty} H_\alpha(TT_3) &= \lim_{\alpha \to \infty} \frac{\log_2\left(\left(\frac{1}{3}\right)^\alpha + \left(\frac{2}{3}\right)^\alpha\right)}{1 - \alpha} \\
&= \lim_{\alpha \to \infty} \frac{\frac{1}{\log 2}\frac{ (\log \frac{1}{3}) (\frac{1}{3})^\alpha + (\log \frac{2}{3}) (\frac{2}{3})^\alpha }{\left(\frac{1}{3}\right)^\alpha + \left(\frac{2}{3}\right)^\alpha}}{-1} \\
&= \lim_{\alpha \to \infty} \frac{1}{\log 2}\frac{ (\log 3) (\frac{1}{3})^\alpha + (\log \frac{3}{2}) (\frac{2}{3})^\alpha }{\left(\frac{1}{3}\right)^\alpha + \left(\frac{2}{3}\right)^\alpha} \\
&= \lim_{\alpha \to \infty} \frac{1}{\log 2} \frac{\log 3 + (\log \frac{3}{2}) 2^\alpha}{1 + 2^\alpha} \\
&= \lim_{u \to \infty} \frac{1}{\log 2} \frac{\log 3 + (\log \frac{3}{2}) u}{1 + u} \\
&= \frac{\log \frac{3}{2}}{\log{2}} \\
&= \log_2 3 - 1 \\
&\approx 0.5850.
\end{align*}

\section{R\'{e}nyi $2$- and $3$-entropy: Min, Max, and What's in Between}

We focus on $H_2$ and $H_3$ on $\mathcal{T}_n$ in this section.  
The results give a strong indication that the R\'{e}nyi $\alpha$-entropy is a measurement of how regular a tournament is, 
similar to \cite{passerini1quantifying}.  
On the other hand, in \cite{landau1951dominance} Landau defined, for 
an $n$-tournament $T$ with score sequence $(s_1, \dots, s_n)$, what he called the \emph{hierarchy} score 
$h(T) = \frac{12}{n^3 - n}\sum_{i = 1}^n \left(s_i - \frac{n-1}{2} \right)^2$;  
this was Landau's measurement of how close $T$ is to the transitive tournament.
It is straightforward to transform $H_2(T)$ into $h(T)$ and vice-versa, given Proposition \ref{prop:renyi2newformula}; hence $H_2$ is equivalent to 
Landau's hierarchy.  
We also enumerate the distinct $H_2$- and $h$-classes, and it can then be seen that $H_2$ and $h$ distinguish tournament structure less than the score sequence does.   The same goes for $H_3$.
But this is not so for $H_{\alpha}$ with $\alpha > 3$; indeed, $H_4$ distinguishes between some $n$-tournaments with the same score sequence for $n \geq 4$. 

Lemma \ref{lem:powerequaltrace} together with equation \ref{number_of_3-cycles} yields the following proposition.

\begin{proposition}\label{prop:renyi2newformula}
  Suppose $T$ is a tournament on vertices $\{v_1, \dots, v_n\}$ with $d^+(v_i) = s_i$.  Then
  \[\sum_{\lambda \in \Lambda_T} \lambda^2 = \binom{n}{2}^{-2}\sum_{i=1}^n s_i^2\]

  and if $c_3(T)$ is the number of 3-cycles in $T$, then 

  \[\sum_{\lambda \in \Lambda_T} \lambda^3= 
  \binom{n}{2}^{-3}\left(\sum_{i =1}^n s_i^3 - 3c_3(T) \right)=
  \binom{n}{2}^{-3}\left(\sum_{i =1}^n s_i^3 - 3\binom{n}{3} + 3\sum_{i=1}^{n}\binom{s_i}{2}\right).\]

\end{proposition}

Define the function $f_k$ on $\mathcal{T}_n$ by $f_k(T) = \sum_{\lambda \in \Lambda_T}\lambda^k$. 

\begin{theorem}\label{thm:minimumpowersums}
On $\mathcal{T}_n$, $f_2$ and $f_3$ are minimized by regular tournaments when $n$ is odd and by nearly-regular tournaments when $n$ is even.
\end{theorem}

\begin{proof}
Consider a tournament $T$ on vertices with score sequence $(s_1, \dots, s_n)$.
Suppose $s_i  + 2 \leq s_j$ for some $i,j$. If $j \rightarrow i$, then construct a new tournament $T'$ by reversing the arc so that $i \rightarrow j$.  
Otherwise, if $i \to j \in A(T)$, consider the tournament $\hat{T}$ induced on $\{i, j\} \cup N^+(j)$.
Note that $j$ is a king in $\hat{T}$, so there is a path $P$ of length $2$ from $j$ to $i$, say $P=(j,u,i)$. Construct $T'$ by reversing the arcs on $P$ so that $i \rightarrow u$ and $u \rightarrow j$ are arcs of $\hat{T}$. 
This reversal lowers the score of $j$ by 1 and increases the score of $i$ by 1, the score of $u$ is unchanged. 
So, in either case, the score sequence of $T'$ is $s_1, \ldots s_i +1, \ldots, s_j -1, \ldots s_n$. It is not difficult to show that 
\begin{equation}\label{eqn:regarcreverse}
    s_i^2 + s_j^2 > (s_i +1)^2 + (s_j -1)^2.
\end{equation}
Let $(s'_1, \ldots s'_n)$ be the score sequence of $T'$,  $E = \spec(\overline{L}_T)$, and $E' = \spec(\overline{L}_T')$. Notice that $s_k = s'_k$ for $k \neq i$ and $k \neq j$. Also $s'_i = s_i +1$ and $s'_j = s_j-1$.  
By Theorem \ref{prop:renyi2newformula}, $ \sum_{\lambda \in E}
\lambda^2= \binom{n}{2}^{-2}\sum_{i =1}^n s_i^2$    
and $\sum_{\lambda \in E'} \lambda^2=
\binom{n}{2}^{-2} \sum_{i = 1}^n (s'_i)^2$. 
These equalities together with equation \ref{eqn:regarcreverse} imply that $\sum_{\lambda \in E} \lambda^2 > \sum_{\lambda \in E'} \lambda^2$.
Repeatedly applying the construction above until there are no scores that differ by at least 2 results in a regular tournament when $n$ is odd and a nearly-regular tournament when $n$ is even. 
(I don't think we need this sentence:) After each step the sum of the squares of the eigenvalues of the resulting tournament is decreased. 

Now consider $\sum_{\lambda \in E'} \lambda^3$, and $T$ and $T'$ are as above with score sequences 
$(s_1, \dots, s_n)$ and $(s_1', \dots, s_n')$, respectively.
By Proposition \ref{prop:renyi2newformula}, we have 
$$\sum_{\lambda \in E'} \lambda^3 = \binom{n}{2}^{-3}\left(\sum_{i =1}^n s_i^3 - 3\binom{n}{3} + 3\sum_{i=1}^{n}\binom{s_i}{2}\right).$$
Consider the part of the sum affected by the algorithm: $(s'_i)^3 + (s'_j)^3 + 3\binom{s'_i}{2} + 3 \binom{s'_j}{2}$.  
Using $s_i + 2 \leq s_j$ (and hence $s_j \geq 2$), $s_i' = s_i+1$, and $s_j-1 = s_j'$, 
the relationship $(s_i')^3 + (s_j)^3 +3\binom{s_i'}{2} + 3\binom{s_j'}{2} < s_i^3 + s_j^3 +3\binom{s_i}{2} + 3\binom{s_j}{2}$ may be obtained.  
Since $\binom{n}{2}^{-3}$ is constant for fixed $n$ as is $3\binom{n}{3}$, the expression for the R\'{e}nyi $3$-entropy will be maximized for small values of $\sum_{i=1}^n s_i^3 - 3\sum_{i=1}^{n}\binom{s_i}{2}$. 
Thus, by changing the scores of $T$ to create a tournament $T'$ in which $s'_j = s_j - 1$ and 
$s'_i = s_i+1$, we see that $H_3(T')>H_3(T)$. 
It follows that the tournament with maximum R\'{e}nyi $3$-entropy will have scores as close to 
equal as possible. 
This is achieved by any regular tournament if $n$ is odd, and any nearly-regular tournament if $n$ is even.
\end{proof}

\begin{corollary} The R\'{e}nyi $2$- and R\'{e}nyi $3$-entropy are maximized by regular $n$-tournaments when $n$ is odd, otherwise by nearly-regular $n$-tournaments.
\end{corollary}

To find the tournaments which minimize the R\'{e}nyi entropy, we use the following algorithm.
Let $T_0$ be a tournament that is not transitive and therefore has a repeated score in its score sequence. For $i \geq 1$, obtain $T_i$ from $T_{i-1}$ by reversing the arc between any pair of vertices with the same score, say $s_m$. Then, if $T_{i-1}$ has score sequence $(s_1, \ldots, s_m, \ldots, s_m, \ldots, s_n)$, then $T_i$ will have score sequence $(s_1, \ldots, s_m - 1, \ldots, s_m + 1, \ldots, s_n)$. Note that $(s_m - 1)^2 + (s_m + 1)^2 = 2s_m^2 + 2.$ Since there are a finite number of $n$-tournaments and each step increases the value of $f_2$ by 2, the algorithm is guaranteed to terminate. This happens $T_i$ has no repeated scores, which is possible only if $T_i$ has score sequence $ (0,1, 2, \dots, n-1)$; that is, $T_i$ is the transitive $n$-tournament.

\begin{theorem} \label{minimumtransitive}
Among all tournaments on $n$ vertices, the R\'{e}nyi 2- and 3-entropy are minimized by the transitive tournament.
\end{theorem}

\begin{proof}
Let $T_0$ be any tournament on $n$ vertices. Apply the algorithm described above until the transitive tournament $TT_n$ is reached. We already established that $f_2$ strictly increases throughout the algorithm, so $H_2(T_0) > H_2(TT_n)$.

It remains to show that $f_3$ does the same.
By Proposition \ref{prop:renyi2newformula}, we have
\begin{align*}f_3(T_i) - f_3(T_{i-1}) &= \left((s_m - 1)^3 + 3\binom{s_m - 1}{2} + (s_m + 1)^3 + 3\binom{s_m + 1}{2}\right) - 2\left(s_m^3 + 3\binom{s_m}{2}\right) \\
&= (s_m + 1)^3 + (s_m - 1)^3 - 2s_m^3 + 3\left[\binom{s_m + 1}{2} - \binom{s_m}{2} + \binom{s_m - 1}{2} - \binom{s_m}{2}\right] \\
&= 6s_m + 3\left[\binom{s_m}{1} - \binom{s_m - 1}{1}\right]\\
&= 6s_m + 3 \\
&\geq 9.
\end{align*}
Indeed, the value of $f_3$ increases by at least $9$ with each step. Therefore, the transitive tournament maximizes $f_3$ and minimizes $H_3$.
\end{proof}

The next and final result in this section gives precisely the number of 
distinct values of $H_2$ on $\mathcal{T}_n$.

\begin{theorem}
For tournaments on $n$ vertices, the number of distinct values of the $H_2$ is 
$$\left\{
\begin{array}{ll}
\frac{1}{4} \binom{n + 1}{3} + 1 & \text{if $n$ is odd,} \\
2\binom{\frac{n}{2} + 1}{3} + 1 & \text{if $n$ is even.}
\end{array} \right.$$
\end{theorem}

\begin{proof}
Using again the algorithm described above with $T_0$ (nearly-)regular and maximizing $f_2$, we take advantage of the fact that each step increases the value of $f_2$ by 2 until the transitive $n$-tournament is reached and $f_2$ is minimized.

Since the sum of the scores of any $T \in \mathcal{T}_n$ is $\binom{n}{2}$, there are an even number of odd scores when $\binom{n}{2}$ is even and an odd number of odd scores when $\binom{n}{2}$ is odd.
Therefore, the sum of the squares of the scores has the same parity as $\binom{n}{2}$. Hence the algorithm produces all possible values of $f_2$.

Now we count the number of values generated by counting the odd or even numbers between minimal and maximal values of $f_2$.
For a transitive tournament, the score sequence is $(0, 1, 2, \ldots, n -1)$, which gives maximum value
$$\sum_{i = 0}^{n - 1} i^2 = \frac{n(n - \frac{1}{2})(n - 1)}{3}.$$
If $n$ is odd, a regular tournament gives minimum value 
\begin{align*}
\sum_{i = 1}^n s_i^2 &= n \left(\frac{n - 1}{2}\right)^2 \\
&= \frac{n(n - 1)^2}{4}
\end{align*}
The number of distinct values for odd $n$ is then 
\begin{align*}
\frac{1}{2}\left(\frac{n(n - \frac{1}{2})(n - 1)}{3} - \frac{n(n - 1)^2}{4}\right) + 1
&= \frac{n(n - 1)}{24} \left(4\left(n - \frac{1}{2}\right) - 3(n - 1)\right) + 1 \\
&= \frac{n(n - 1)(n + 1)}{24}  + 1 \\
&= \frac{1}{4}\binom{n + 1}{3} + 1.
\end{align*}
If $n$ is even, a nearly-regular tournament has $\frac{n}{2}$ vertices with score $\frac{n}{2} - 1$ and $\frac{n}{2}$ vertices with score $\frac{n}{2}$, so
\begin{align*}
\sum_{i = 1}^n s_i^2 &= \frac{n}{2}\left(\frac{n}{2} - 1\right)^2 + \frac{n}{2}\left(\frac{n}{2}\right)^2 \\
&= \frac{n\left((n - 2)^2 + n^2\right)}{8} \\
&= \frac{n(n^2 - 2n + 2)}{4}.
\end{align*}
Therefore, the number of distinct values for even $n$ is
\begin{align*}
&\frac{1}{2}\left(\frac{n(n - \frac{1}{2})(n - 1)}{3} - \frac{n(n^2 - 2n + 2)}{4}\right) + 1 \\
&= \frac{n}{24}\left(4\left(n - \frac{1}{2}\right)(n - 1) - 3(n^2 - 2n + 2)\right) +1 \\
&= \frac{n(n^2 - 4)}{24} + 1 \\
&= \frac{\frac{n}{2}(\frac{n}{2} - 1)(\frac{n}{2} + 1)}{3} + 1 \\
&= 2\binom{\frac{n}{2} + 1}{3} + 1.
\end{align*}
\end{proof}

Let $h_n^{\alpha}$ be the number of distinct values for $H_{\alpha}$ over $\mathcal{T}_n$, and 
$S_n$ denote the number of distinct score sequences of $n$-tournaments in $\mathcal{T}_n$.  
The table below shows $h_n^2$ and $S_n$ up to $n=10$.  
$S_n$ is sequence A000571 in the OEIS \cite{OEIS}.

\begin{equation}
\begin{array}{| c | c | c | c | c|c|c|c|c|c|}
\hline 
n & 2 & 3 &4 & 5 & 6 & 7 & 8 & 9 & 10 \\
\hline
S_n & 1 & 2 & 4 & 9 & 22 & 59 & 167 & 490 & 1486 \\
h^2_n & 1 & 2 & 3 & 6 & 9 & 15 & 21 & 31 & 41 \\

\hline
\end{array}\label{table:number_score_sequences_vs_H2}
\end{equation}

We have observed that, as $\alpha$ increases, $h_n^{\alpha}/S_n$ increases and we make the following conjecture.

\begin{conjecture} For $\alpha$ sufficiently large, $\ds \lim_{n \to \infty}\frac{h_n^{\alpha}}{S_n} > 1$. 
\end{conjecture}

\section{R\'{e}nyi $4$-entropy}

In this section we focus on $\alpha = 4$ and regular $n$-tournaments for $n>5$. 
Recall that the \emph{out-set} of a vertex $v$ is the set of vertices at the heads of arcs whose tail is at $v$.

For any $n$ there is up to isomorphism a unique transitive tournament on $n$ vertices, but the case is different for regular tournaments.  For example there are $1, 3, 15, 1223$, and $1,495,297$ regular $n$-tournaments for $n=5, 7, 9, 11$, and $13$, respectively. 
Let $\mathcal{R}_n$ denote the set of regular tournaments in $\mathcal{T}_n$.
The results of the previous section showed that regular and nearly-regular tournaments maximize the R\'{e}nyi $\alpha$-entropy for $\alpha = 2$ and $\alpha =3$.
If $\alpha >3$, what can be said about $H_{\alpha}(T)$? 
What we have seen experimentally is that $H_{\alpha}(T)$ is among the largest values of 
$H_{\alpha}$ on $\mathcal{T}_n$ if $T \in \mathcal{R}_n$; that is,
if $T \in \mathcal{R}_n$ and $T' \in \mathcal{T}_n \setminus \mathcal{R}_n$, then $H_{\alpha}(T') < H_{\alpha}(T)$.
What we have proved is that $H_4$ partitions $\mathcal{R}_n$, and it is this effect we explore presently.
For example, there are three regular $7$-tournaments, $QR_7$, $B$, and $R_7$ drawn in Figure \ref{fig:regular_7}, and $H_4$ gives a distinct value to each:
$$H_4(R_7) < H_4(B) < H_4(QR_7).$$

\begin{figure}[ht]
\begin{center}
	\begin{tikzpicture}[scale = .4, >={Latex[length=4mm, width=1.5mm]}]
			\foreach \y/\z in {A/0,B/13,C/26}{
			\foreach \x in {0,...,6}{
			\begin{scope}[xshift=\z cm]
    			\draw[fill] (90-\x*51.42: 4) circle (5pt) node (\y\x) {};
    			\node at (90-\x*51.42:4.75) {$\x$};
			\end{scope}
			}}
			
			\node[below] at (0,-4.5) {$QR_7$};
			\draw[->] (A0) -- (A1);
			\draw[->] (A0) -- (A2);
			\draw[->] (A0) -- (A4);
			\draw[->] (A1) -- (A2);
			\draw[->] (A1) -- (A3);
			\draw[->] (A1) -- (A5);
			\draw[->] (A2) -- (A3);
			\draw[->] (A2) -- (A4);
			\draw[->] (A2) -- (A6);
			\draw[->] (A3) -- (A4);
			\draw[->] (A3) -- (A5);
			\draw[->] (A3) -- (A0);
			\draw[->] (A4) -- (A5);
			\draw[->] (A4) -- (A6);
			\draw[->] (A4) -- (A1);
			\draw[->] (A5) -- (A6);
			\draw[->] (A5) -- (A0);
			\draw[->] (A5) -- (A2);
			\draw[->] (A6) -- (A0);
			\draw[->] (A6) -- (A1);
			\draw[->] (A6) -- (A3);
			
			\node[below] at (13,-4.5) {$B$};
		    \draw[->] (B0) -- (B3);
			\draw[->] (B0) -- (B4);
			\draw[->] (B0) -- (B6);
			\draw[->] (B1) -- (B0);
			\draw[->] (B1) -- (B6);
			\draw[->] (B1) -- (B5);
			\draw[->] (B2) -- (B0);
			\draw[->] (B2) -- (B1);
			\draw[->] (B2) -- (B6);
			\draw[->] (B3) -- (B1);
			\draw[->] (B3) -- (B2);
			\draw[->] (B3) -- (B5);
			\draw[->] (B4) -- (B1);
			\draw[->] (B4) -- (B2);
			\draw[->] (B4) -- (B3);
			\draw[->] (B5) -- (B0);
			\draw[->] (B5) -- (B2);
			\draw[->] (B5) -- (B4);
			\draw[->] (B6) -- (B3);
			\draw[->] (B6) -- (B4);
			\draw[->] (B6) -- (B5);
			
			\node[below] at (26,-4.5) {$R_7$};
		    \draw[->] (C0) -- (C1);
			\draw[->] (C0) -- (C2);
			\draw[->] (C0) -- (C3);
			\draw[->] (C1) -- (C2);
			\draw[->] (C1) -- (C3);
			\draw[->] (C1) -- (C4);
			\draw[->] (C2) -- (C3);
			\draw[->] (C2) -- (C4);
			\draw[->] (C2) -- (C5);
			\draw[->] (C3) -- (C4);
			\draw[->] (C3) -- (C5);
			\draw[->] (C3) -- (C6);
			\draw[->] (C4) -- (C5);
			\draw[->] (C4) -- (C6);
			\draw[->] (C4) -- (C0);
			\draw[->] (C5) -- (C6);
			\draw[->] (C5) -- (C0);
			\draw[->] (C5) -- (C1);
			\draw[->] (C6) -- (C0);
			\draw[->] (C6) -- (C1);
			\draw[->] (C6) -- (C2);

		\end{tikzpicture}
\end{center}
\caption{The regular $7$-tournaments}\label{fig:regular_7}
\end{figure}
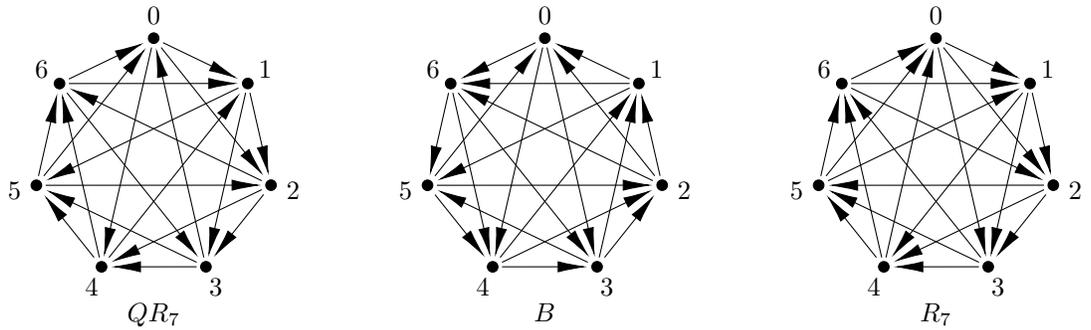

The regular $7$-tournaments $QR_7$ and $R_7$ are distinguishable in several ways; for example, the out-set of every vertex in $QR_7$ induces the $3$-tournament $C_3$ of Figure \ref{fig:3_tourns}, while every out-set 
of $R_7$ induces $TT_3$ of Figure \ref{fig:3_tourns}.  
$QR_7$ and $R_7$ are examples of two classes of tournaments that will be of interest in this section. 
For the next two definitions, suppose the $n$-tournaments have vertex set $\{0,1,2,\dots, n-1\}$.
Let $S \subset \{0,1,2,\dots, n-1\}$ with $|S| = \frac{n-1}{2}$ and $i-j \neq 0$ modulo $n$ for all $i,j \in S$.  
An $n$-tournament $T$ is \emph{rotational with symbol $S$}, if $i \to j$ in $T$ if and only if $j - i \in S$. 
A \emph{doubly regular} $n$-tournament $T$ is a regular tournament with the additional property that 
for any two vertices $x, y \in V(T)$, $|N^+(x) \cap N^+(y)| = k$; necessarily $n =4k+3$. 
Equivalently a doubly-regular $(4k+3)$-tournament is a regular tournament in which the out-set of each vertex induces a 
regular $(2k+1)$-tournament.  
$QR_7$ is doubly regular and is the rotational $7$-tournament with symbol $\{1,2,4\}$, the nonzero quadratic residues modulo $7$. 
$R_7$ is the rotational $7$-tournament with symbol $\{1,2,3\}$. 
We also indentify the following class of tournaments. 
A \emph{quasi doubly regular} tournament on $4k + 1$ vertices is a regular tournament with score of each vertex equal to $2k$ and, for any pair of vertices $x$ and $y$, $|N^+(x) \cap N^+(y)| \in \{k - 1, k\}$.

For simplicity, and since the log function is an artifact of what was desired out of an entropy function (see \cite{renyi1961measures}), we focus on the power sums of the eigenvalues, and define, for a tournament or any directed or undirected graph $T$, $$H^*_{\alpha}(T) = -f_{\alpha}(T)= - \sum_{\lambda \in \mathrm{spec}(\overline{L}(T))} \lambda^\alpha.$$
Note that minimizing $H^*_{\alpha}$ maximizes $H_{\alpha}$ when $H_{\alpha}$ is defined.

We first show that $H^*_4(T)$ is minimum on $\mathcal{R}_n$ if and only if $T$ is quasi doubly regular or doubly regular if $n = 4k+1$ or $n = 4k+3$, respectively.  
We'll use the following lemma which counts the number of distinct subtournaments isomorphic to $TS_4$ of Figure \ref{fig:4_tournaments}. 

Let $T$ be an $n$-tournament and define: 
\begin{itemize}
    \item $c_3(T)$ to be the number of subtournaments isomorphic to $C_3$ of Figure \ref{fig:3_tourns};
    \item $c_4(T)$ to be the number of subtournaments of $T$ isomorphic to $TS_4$ of Figure \ref{fig:4_tournaments} -- the strongly connected\footnote{A digraph is \emph{strongly connected} if between any pair of vertices $x$ and $y$ there is a path from $x$ to $y$ and a path from $y$ to $x$.} $4$-tournament;
    \item $t_4(T)$ to be the number of subtournaments of $T$ isomorphic to $TT_4$ of Figure \ref{fig:4_tournaments} -- the transitive $4$-tournament.
\end{itemize}  

We note that the following lemma addresses a problem similar to that in \cite{linial2016number} their Proposition 1.1).

\begin{lemma} \label{lemma:c_3_c_4_t_4}
Let $T$ be an $n$-tournament on $n$ vertices, and $c_3$, $c_4$, and $t_4$ defined as above; then $$c_4(T) = t_4(T) - \frac{n - 3}{4} \left(\binom{n}{3} - 4c_3(T)\right).$$
\end{lemma}

\begin{proof}
Consider the four $4$-tournaments up to isomorphism:
\begin{enumerate}
\item $TS_4$: The strong 4-tournament;
\item $TT_4$: The transitive 4-tournament;
\item $TO_4$: The tournament with score sequence (1, 1, 1, 3);
\item $TK_4$: The tournament with score sequence (0, 2, 2, 2).
\end{enumerate}
It is quickly verified that
$$c_3(C_4) = 2, \quad c_3(T_4) = 0, \quad c_3(TK_4) = 1, \quad c_3(TO_4) = 1.$$
Now let $T$ be any $n$-tournament. 
Since each $3$-cycle (subtournament isomorphic to $C_3$) belongs to exactly $n - 3$ subtournaments of $T$ on $4$ vertices, we have
\begin{equation}
(n - 3)c_3(T) = 2c_4(T) + to_4(T) + tk_4(T),
\end{equation}
where $to_4(T)$ and $tk_4(T)$ are the number of $TO_4$'s and $TK_4$'s in $T$.
Furthermore, the total number of subtournaments of $T$ on $4$ vertices is equal to
\begin{equation}
c_4(T) + t_4(T) + to_4(T) + tk_4(T) = \binom{n}{4}.
\end{equation}
Combining equations (2) and (3), we obtain
\begin{align*}
c_4(T) &= t_4(T) - \binom{n}{4} + (n - 3)c_3(T) \\
&= t_4(T) - \frac{n - 3}{4}\left(\binom{n}{3} - 4c_3(T)\right).
\end{align*}

\end{proof}

\begin{lemma}\label{lemma:H^*_4_max_when_t_4_min}
For regular tournaments, $H_4^*(T)$ is maximized where $t_4(T)$ is minimized, and vice versa.
\end{lemma}
\begin{proof}
Let $T = (V, A)$ be a regular tournament on $n = 2m + 1$ vertices. First note that for $\alpha \in \mathbb Z$ with $\alpha \geq 2$,  we have $H^*_\alpha(T) = -\tr(\bar L(T)^\alpha)$. 
Furthermore, since $T$ is regular, we have
$$\bar L(T) = \frac{1}{\binom{n}{2}}(mI - M).$$
Therefore, by the linearity of the trace and using Lemma \ref{lemma:c_3_c_4_t_4}, we can express $H^*_4$ in terms of $t_4(T)$, noting that $TS_4$ is the only tournament on $4$ vertices with a walk of length $4$ from a vertex to itself.
\begin{align*}
H^*_4(T) &= -\binom{n}{2}^{-4}\text{Tr}\left(m^4I - 4m^2M + 6m^2M^2 - 4mM^3 + M^4\right) \\
&= -\binom{n}{2}^{-4}\left(m^4n - 12mc_3(T) + 4c_4(T)\right) \\
&= -\binom{n}{2}^{-4}\left(m^4n - 12mc_3(T) + 4t_4(T) - (n - 3 )\left(\binom{n}{3} - 4c_3(T)\right)\right).
\end{align*}
Note that  $n$, $m$ and $c_3$ are all constant for regular tournaments on $n$ vertices.
\end{proof}

We next identify the regular tournament which minimizes $H_4$ on $\mathcal{R}_n$; it is a rotational tournament.  
A rotational tournament is distinguished by its symbol $S$, and we 
call the rotational tournament with symbol $S = \left\{1,2,\dots, \frac{n-1}{2}\right\}$ the \emph{consecutive rotational $n$-tournament}.   

\begin{theorem}
On $\mathcal{R}_{2m+1}$, $H_4(T)$ is minimum if and only if $T$ is isomorphic to the consecutive rotational tournament.
\end{theorem}

\begin{proof}
Let $T$ be a regular tournament on $n = 2m + 1$ vertices. 
By Lemma \ref{lemma:H^*_4_max_when_t_4_min}, we look to maximize $t_4(T)$. Since each vertex has score $m$, each vertex is the source of at most $\binom{m}{3}$ $TT_4$'s, and this value is achieved if and only if the outset of that vertex is transitive. 
If each of the vertices in $T$ have this property, then the maximum value $n\binom{m}{3}$ of $t_4(T)$ is achieved. 
For each odd $n$, there is only one such tournament up to isomorphism, namely the consecutive rotational tournament.

To see this, let $N^+(x)$ be transitive for each $x \in V(T),$ and relabel the vertices the following way in $\mathbb Z_n$. 
Choose a vertex to label $0$. 
Label the source of $N^+(0)$ by $1$, the source of $N^+(0) \cap N^+(1)$ by $2$, and so on until $N^+(0)$ consists of $\{1, 2, \ldots, m\}$. 
Then $m$ is beaten by $0, \ldots, m - 1$, so $m$ must beat all of the remaining vertices, with $N^+(m)$ transitive. 
Label the source of $N^+(m)$ by $m + 1$, the source of $N^+(m) \cap N^+(m + 1)$ by $m + 2$, and so on until all of the vertices are labeled $0, \ldots, n - 1$. 
Now $m$ beats $m + 1, \ldots, n - 1, 0$, so $2, \ldots, m$ must beat $m + 1$.
Then $1$ beats $2, \ldots, m + 1$, so $m + 2, \ldots, n - 1, 0$ must beat $1$. 
This means that $m + 2$ beats $m + 3, \ldots n - 1, 0, 1$, so $2, \ldots, m + 1$ must beat $m + 2$. 
Continuing in this fashion, we see that for vertices $x$ and $y$, $x \to y$ if and only if $y - x \in \{1, 2, \ldots, m\}$, so $T$ is isomorphic to the consecutive rotational $n$-tournament.
\end{proof}

We now find the argument maximum of $H_4$ on $\mathcal{R}_{n}$.

\begin{theorem}
A $(4k+3)$-tournament $T$ achieves the maximum value of $H_4$ on $\mathcal{R}_{4k+3}$ if and only if $T$ is doubly regular. 
A $(4k+1)$-tournament $T$ achieves the maximum value of $H_4$ on $\mathcal{R}_{4k+1}$ if and only if $T$ is quasi doubly regular.
\end{theorem}

\begin{proof}
Let $T$ be a regular tournament on $n = 2m + 1$ vertices.
Now we look to minimize $t_4(T)$. 
Consider a vertex $x \in V(T)$ and the corresponding subtournament $T'$ on the $m$ vertices in $N^+(x)$. 
The number of transitive triples in $T'$ is given by
\begin{align*}
t_3(T') &= \sum_{y \in N^+(x)} \binom{|N^+(x) \cap N^+(y)|}{2} \\
&= \frac{1}{2}\sum_{y \in N^+(x)}\left(|N^+(x) \cap N^+(y)| - \frac{m - 1}{2}\right)^2 \\
& \qquad + \frac{(m - 2)}{2}\sum_{y \in N^+(x)}|N^+(x) \cap N^+(y)| \;\;-\;\; \frac{1}{2}\!\!\sum_{y \in N^+(x)} \left(\frac{m - 1}{2}\right)^2 \\
&= \frac{1}{2}\left(\sum_{y \in N^+(x)}\left(|N^+(x) \cap N^+(y)| - \frac{m - 1}{2}\right)^2 + (m - 2)\binom{m}{2} - m\left(\frac{m - 1}{2}\right)^2\right).
\end{align*}

If $n \equiv 3 \pmod 4$ and $n = 4k + 3$, then 
\begin{align*}
t_3(T') &\geq \frac{1}{2}\left((m - 2)\binom{m}{2} - m\left(\frac{m - 1}{2}\right)^2\right)\\
&= \frac{m}{2}\left((2k - 1)k - k^2\right) \\
&= m\frac{2k^2 - k - k^2}{2} \\
&= m\binom{k}{2},
\end{align*}
with equality if and only if $|N^+(x) \cap N^+(y)| = \frac{m - 1}{2} = k$ for each $y \in N^+(x)$.
Now, since $t_3(T')$ is also the number of $T_4$s in $T$ in which $x$ is the source, it follows that $$t_4(T) \geq nm\binom{k}{2},$$
with equality if and only if $T$ is doubly regular.

If $n \equiv 1 \pmod 4$ and $n = 4k + 1$, then
\begin{align*}
t_3(T') &\geq \frac{1}{2}\left(m\left(\frac{1}{2}\right)^2 + (m - 2)\binom{m}{2} - m\left(\frac{m - 1}{2}\right)^2\right) \\
&= \frac{m}{2}\left(\frac{1}{4} + (2k - 2)\frac{2k - 1}{2} - \left(k - \frac{1}{2}\right)^2\right) \\
&= k\left((k - 1)(2k - 1) - k^2 + k\right) \\
&= k(k - 1)^2,
\end{align*}
with equality if and only if $\left|N^+(x) \cap N^+(y)| - \frac{m - 1}{2}\right| = \frac{1}{2}$ for each $y \in N^+(x)$.
Therefore,
$$t_4(T) \geq nk(k - 1)^2,$$
with equality if and only if $T$ is quasi doubly regular.
\end{proof}

From this, we obtain the tight bounds for regular tournaments
$$-\frac{n(n - 1)\left(3n^3 - 17n^2 + n - 3\right)}{48} \leq \binom{n}{4}^4H_4^*(T) \leq -\frac{n^2(n-1)(n^2 - 6n + 1)}{16} \quad \text{for } n \equiv 3 (\text{mod } 4);$$
$$-\frac{n(n - 1)\left(3n^3 - 17n^2 + n - 3\right)}{48} \leq \binom{n}{4}^4H_4^*(T) \leq -\frac{n(n - 1)^2(n^2 - 5n - 4)}{16} \quad \text{for } n \equiv 1 (\text{mod } 4).$$

Since $H_\alpha$ depends entirely on the spectrum, we know that as $\alpha$ increases there is no partitioning of $\mathcal{R}_n$ via $H_{\alpha}$ beyond the spectrum-level.
The next result shows that all doubly regular $n$-tournaments have the same spectrum. 

\begin{theorem}
For any doubly regular tournament $T$ on $n = 2m + 1 = 4k + 3$ vertices,
$$\text{spec}(\bar L(T)) = \left\{0, \frac{1}{n - 1}\left(1 \pm \frac{i}{\sqrt n}\right)^{(m)}\right\},$$ where $(m)$ in superscript denotes that the eigenvalue has multiplicity $m$.
\end{theorem}
\begin{proof}
Let $A$ be the adjacency matrix of a doubly-regular tournament $T$ on $n = 2m + 1 = 4k + 3$ vertices. Then $A A^\trans = mI + k(J - I)$ and $A + A^\trans = J - I$, where $I$ is the identity matrix and $J$ is the all-ones matrix. Then 
\begin{align*}
(A - \lambda I)(A - \lambda I)^\trans &= A A^\trans - \lambda(A + A^\trans) + \lambda^2 I \\
&= mI + k(J - I) - \lambda(J - I) + \lambda^2 I \\
&= (k - \lambda)J + (m - k + \lambda + \lambda^2)I.
\end{align*}

Since $spec(J) = \{n, 0^{(n - 1)}\}$, we have $$spec((A - \lambda I)(A - \lambda I)^\trans) = \{n(k - \lambda) + m - k + \lambda + \lambda^2, (m - k + \lambda + \lambda^2)^{(n - 1)}\}.$$ Therefore, 
\begin{align*}
|A - \lambda I|^2 &=
|(A - \lambda I)(A - \lambda I)^\trans| \\
&= (n(k - \lambda) + m - k + \lambda + \lambda^2) (m - k + \lambda + \lambda^2)^{n - 1} \\
&= (m^2 - 2m\lambda + \lambda^2) (k + 1 + \lambda + \lambda^2)^{n - 1} \\
&= \left((m - \lambda)(k + 1 + \lambda + \lambda^2)^m\right)^2.
\end{align*}
Therefore, since $n$ is odd, $$|A - \lambda I| = (m - \lambda)(k + 1 + \lambda + \lambda^2)^m$$
and
$$spec(A) = \left\{m, \left(-\frac{1}{2} \pm \frac{i\sqrt n}{2}\right)^{(m)}\right\}.$$
Finally, if $\bar L$ is the normalized Laplacian matrix of $T$, then $\bar L$ and $A$ are related by
$$\bar L = \frac{2}{n(n - 1)}(mI - A),$$ so 
\begin{align*}
spec(\bar L) &= \left\{\frac{2}{n(n - 1)}(m - m), \frac{2}{n(n - 1)}\left(m + \frac{1}{2} \pm \frac{i\sqrt n}{2}\right)^{(m)}\right\} \\
&= \left\{0, \frac{1}{n - 1}\left(1 \pm \frac{i}{\sqrt n}\right)^{(m)}\right\}.
\end{align*}
\end{proof}

\begin{corollary} For integer $\alpha >4$, $H_{\alpha}$ is maximized on $\mathcal{R}_{4k+3}$ via doubly regular tournaments.   
\end{corollary}

\section{Von Neumann Entropy and Random Walks}

We believe we have a compelling argument that, as far as directed graphs are concerned, the  R\'{e}nyi entropy entropy calculation quantifies the \emph{regularity} of the directed graph.  
Entropy is apparently sensitive to local regularity vis-\'{a}-vis the refinement of the R\'{e}nyi ordering we observe on the set of regular tournaments with highest entropy being associated to doubly-regular tournaments, tournaments that are regular and locally regular. 
But this is either saying nothing, given that `regularity' has not been precisely defined, or we are simply defining `regularity' 
as \emph{the extent to which entropy is high relative to other directed graphs}. 

In this section we more precisely describe what the von Neumann entropy calculation is quantifying in graphs and directed graphs. 
First we establish a lemma about the magnitudes of the eigenvalues of the scaled Laplacian.
Let $L$ and $\overline{L}$ be the Laplacian and normalized Laplacian of some loopless directed graph $\Gamma$ on the set of $n$ vertices $\{v_1, \dots, v_n\}$,
$\{\lambda_1, \dots, \lambda_n\}$ the multiset of eigenvalues of  
$\overline{L}$, and let $d_i^+ = d_{\Gamma}^+(v_i)$ denote the out-degree of vertex $v_i$ in $\Gamma$.

The next lemma is established to the end of supporting the following extension of the von Neumann entropy to a directed or undirected graph:
Suppose $\overline{L}$ is the Laplacian of a (di)graph $\Gamma$ normalized as in this paper, then the \emph{von Neumannn entropy} of $\Gamma$ is  
$$S(\vec \lambda) = \frac{1}{\log 2}\left(\text{tr}(\overline{L}) - \sum_{j = 2}^\infty \frac{\text{tr}(\overline{L}^j)}{j(j - 1)}\right).$$

\begin{lemma}\label{Lem:Markov} Regarding $\Gamma, L, \overline{L}$, and $\{\lambda_1, \dots, \lambda_n\}$ as described above:  
$|\lambda_k - 1| \leq 1$ for $1 \leq k \leq n$.
\end{lemma}

\begin{proof}
Consider the family $\mathcal F$ of matrices of the form
\begin{equation*}
M = \left(I - S L\right)^t \qquad \text{with} \qquad S = 
\begin{bmatrix}
\frac{1}{s_1} & 0 & \cdots & 0 \\
0 & \frac{1}{s_2} & & 0 \\
\vdots & &\ddots & \vdots \\
0 & 0 & \cdots & \frac{1}{s_n}
\end{bmatrix},
\end{equation*}
where $d^+_i \leq s_i \neq 0$ for all $i$.

Since the row sums of $L$ are all zero, and $S$ scales each row of $L$ individually, the same is true of the rows of $S L$. 
Therefore, the row sums of $M^\trans$, and consequently the column sums of $M$, equal $1$.

Furthermore, note that all elements of $M$ are between $0$ and $1$.
On the diagonal, the $k^\text{th}$ element is $1 - d^+_k/s_k$.
Off the diagonal, each element is either $0$ or $1/s_k$ for some $s_k$.
Hence $M$ is a Markov matrix, which guarantees that each of its eigenvalues has modulus at most $1$.

Notice in particular that $\overline{L}$ is of the form $S L$ with $s_1 = s_2 = \cdots = s_n = \sum_{i=1}^n d_i^+$. 
Now suppose that $\lambda$ is an eigenvalue of $\overline{L}$. 
Then $\lambda$ is also an eigenvalue of $\overline{L}^\trans$, so $1 - \lambda$ is an eigenvalue of $M = I - \overline{L}^\trans = (I - \overline{L})^\trans$, where $M \in \mathcal F$.
This means that $1 - \lambda$ has modulus at most 1.
\end{proof}

Recall that the function 
\begin{equation*}
f(\lambda) = \left\{\begin{array}{l l}
\lambda \log_2 \frac{1}{\lambda} & \text {if }\lambda \neq 0 \\
0 & \text{if } \lambda = 0
\end{array}\right.
\end{equation*}
can be expanded as the power sum
$$f(\lambda) = \frac{1}{\log 2}\left((1 - \lambda) - \sum_{j = 2}^\infty \frac{(1 - \lambda)^j}{j(j - 1)}\right)\qquad \text{for $|\lambda - 1| \leq 1.$}$$

By Lemma \ref{Lem:Markov} the eigenvalues of the scaled Laplacian matrix are all within the radius of convergence of $f$ and the von Neumann entropy can be expressed as 
\begin{align*}
S(\vec \lambda) &= \sum_{k = 1}^n f(\lambda_k) \\
&= \sum_{k = 1}^n \frac{1}{\log 2}\left((1 - \lambda_k) - \sum_{j = 2}^\infty \frac{(1 - \lambda_k)^j}{j(j - 1)}\right) \\
&= \frac{1}{\log 2}\left(\sum_{k = 1}^n (1 - \lambda_k) - \sum_{j = 2}^\infty \frac{1}{j(j - 1)}\sum_{k = 1}^n (1 - \lambda_k)^j\right).
\end{align*}

\noindent Also, by Lemma \ref{Lem:Markov}, we know that $\{1 - \lambda_k\}_{k = 1}^n$ is the spectrum of the Markov matrix $M = (I - \bar L_{\Gamma})^\trans$.
Therefore, for $j \geq 1$,
$$\sum_{k = 1}^n (1 - \lambda_k)^j = \text{tr}(M^j),$$
and since  $\lim_{j \to \infty} \text{tr}(M^j) = 1$, because $M$ is a Markov matrix, we can write 
$$S(\vec \lambda) = \frac{1}{\log 2}\left(\text{tr}(M) - \sum_{j = 2}^\infty \frac{\text{tr}(M^j)}{j(j - 1)}\right).$$\\

%

Let $g$ denote the sum $\sum d_i^+$ of out-degrees of vertices of $\Gamma$.
Let $w_j(v_k)$ be a random walk of length $j$ starting at vertex $v_k$, 
where at each step, 
the walk has a probability of $1/g$ of moving to each vertex in its out-set.
Then entry $l,k$ of matrix $M^j$ is the probability that $w_j(v_k)$ ends at $v_l$, and
$$\text{tr}(M^j) = \sum_{k = 1}^n P(\text{$w_j(v_k)$ ends at $v_k$}).$$

Therefore, the von Neumann entropy can be expressed as
$$S(\vec \lambda) = \frac{1}{\log 2}\left(n - 1 - \sum_{k = 1}^n \sum_{j = 2}^\infty \frac{P(\text{$w_j(v_k)$ ends at $v_k$})}{j(j - 1)}\right).$$
In this sense, the von Neumann entropy is a measure of how quickly a random walk will move away from its initial state and settle in to its limiting state.

Also, this viewpoint allows us to place general bounds on the von Neumann entropy. 
\begin{observation} \label{entropy bounds}
For any loopless directed graph $\Gamma$, $S(\Gamma) \leq S(\vec d^+)$, where
\[\vec d^+ = (d^+(v_1)/g, \ldots, d^+(v_n)/g)\]
is the distribution of out-degrees in $\Gamma$, and equality holds if and only if $\Gamma$ has no (directed) cycles.
\end{observation}
\begin{proof}
Clearly, $$P(\text{$w_j(v_k)$ ends at $v_k$}) \geq P(\text{$w_j(v_k)$ never leaves $v_k$}) = (1 - d^+_k/g)^j,$$ with equality if and only if $\Gamma$ has no directed cycles.
Therefore,
\begin{align*}
S(\vec \lambda) &\leq \frac{1}{\log 2} \left(n - 1 - \sum_{k = 1}^n \sum_{j = 2}^\infty \frac{(1 - d^+(v_k)/g)^j}{j(j - 1)}\right) \\
&= \frac{1}{\log 2}\left(n - 1 - \sum_{k = 1}^n \left(\frac{d^+(v_k)}{g} \log \frac{d^+(v_k)}{g} + 1 - \frac{d^+(v_k)}{g}\right)\right) \\
&= -\sum_{k = 1}^n \frac{d^+(v_k)}{g} \log_2 \frac{d^+(v_k)}{g} \\
&= S(\vec d^+).
\end{align*}
\end{proof}
Note that the condition for equality is equivalent to $\bar L_{\Gamma}$ being permutation equivalent to an upper-triangular matrix. 
This makes sense, since in that case the eigenvalues of $L_{\Gamma}$ are the out-degrees of the vertices of $\Gamma$.

\begin{corollary}
For any loopless directed graph $\Gamma$, $S(\Gamma) < \log_2 n$.
\end{corollary}
\begin{proof}
Since the out-degrees are real-valued, we have $S(\Gamma) = S(\vec \lambda) \leq S(\vec d^+) \leq \log_2 n$. 
If $S(\vec d^+) = \log_2 n$, then $d^+_1 = \ldots = d^+_n > 0$, and $\Gamma$ must have a directed cycle, so $S(\vec \lambda) < S(\vec d^+)$.
\end{proof}


\bibliography{bib}{}
\bibliographystyle{plain}
\end{document}